\documentclass[11pt]{article}

\usepackage[mathscr]{eucal}
\usepackage{color}
\usepackage{amsmath}
\usepackage{amssymb}
\usepackage{amsfonts}
\usepackage{amscd}
\usepackage{amsthm}
\usepackage{latexsym}
\usepackage{graphicx}
\usepackage{subfig}

\def\be{\begin{equation}}
\def\ee{\end{equation}}
\def\bse{\begin{subequations}}
\def\ese{\end{subequations}}

\usepackage{latexsym}

\let\er\eqref

\let\be\beta

\newcommand{\R}{{\mathbb R}}

\newtheorem{theorem}{Theorem}
\newtheorem{lemma}[theorem]{Lemma}
\newtheorem{proposition}[theorem]{Proposition}
\newtheorem{remark}[theorem]{Remark}
\newtheorem{corollary}[theorem]{Corollary}
\newtheorem{definition}[theorem]{Definition}

\def\bse{\begin{subequations}}
\def\ese{\end{subequations}}

\usepackage{geometry}

\title{The aggregation-diffusion equation with energy critical exponent}

\author{ Shen Bian\footnote{Corresponding author. Beijing University of Chemical Technology, 100029, Beijing. Email: \texttt{bianshen66@163.com}. } }
\date{}

\begin{document}
\let\cleardoublepage\clearpage

\maketitle

\begin{abstract}
  We consider a Keller-Segel model with non-linear porous medium type diffusion and nonlocal attractive power law interaction, focusing on potentials that are less singular than Newtonian interaction. Here, the nonlinear diffusion is chosen to be $m=\frac{2d}{d+2s}$ in such a way that the associated free energy is conformal invariant and there is a family of stationary solutions $U(x)=c\left(\frac{\lambda}{\lambda^2+|x-x_0|^2}\right)^{\frac{d+2s}{2}}$ for any constant $c$ and some $\lambda>0, x_0 \in \R^d.$ We analyze under which conditions on the initial data the regime that attractive forces are stronger than diffusion occurs and classify the global existence and finite time blow-up of dynamical solutions by virtue of stationary solutions. Precisely, solutions exist globally in time if the $L^m$ norm of the initial data $\|u_0\|_{L^m(\R^d)}$ is less than the $L^m$ norm of stationary solutions $\|U(x)\|_{L^m(\R^d)}$. Whereas there are blowing-up solutions for $\|u_0\|_{L^m(\R^d)}>\|U(x)\|_{L^m(\R^d)}$.
\end{abstract}

\section{Introduction}

In the present paper, we analyse qualitative properties of nonnegative solutions for the aggregation-diffusion equation
\begin{align}\label{fracks}
\left\{
      \begin{array}{ll}
      u_t = \Delta u^m-\nabla \cdot \left(u \nabla c\right), ~~& x\in \R^{d},~t> 0,\\
     u(x,0) =u_{0}(x) \ge 0, ~~& x\in \R^d.
      \end{array}\right.
\end{align}
Here the diffusion exponent is taken to be $m=\frac{2d}{d+2s}$ for $2<2s<d$ with $d \ge 3.$ $u$ represents the density of cells, the chemoattractant $c$ is governed by a diffusion process and can be expressed by the convolution of Riesz potential and $u(x,t)$, that is
\begin{align}\label{CC}
c(x,t) = \frac{1}{d-2s} \int_{\mathbb R^d} \frac{u(y,t)}{|x-y|^{d-2s}} dy.
\end{align}
The parabolic equation on $u$ gives nonnegative solutions $u(x,t)\ge 0, c(x,t)\ge 0$. Another property we will use is the conservation of the total number of cells
\begin{align}
M:=\int_{\R^d} u_0(x)dx=\int_{\R^d}u(x,t) dx.
\end{align}
Initial data will be assumed throughout this paper
\begin{align}\label{12}
u_0 \in L_+^1 \cap L^\infty(\R^d), \quad \int_{\R^d} |x|^2 u_0(x) dx<\infty, \quad \nabla u_0^m \in L^2(\R^d).
\end{align}

The nonlinear diffusion $m>1$ models the local repulsion of cells while the attractive nonlocal term models cell movement toward chemotactic sources or attractive interaction between cells due to cell adhesion \cite{BL98,KS70,PB15,FW03,MT15,GC08}. The main characteristic of equation \er{fracks} is the competition between the diffusion and the nonlocal aggregation. This is well presented by the free energy
\begin{align}\label{Fu}
F(u)=\frac{1}{m-1} \int_{\R^d} u^m dx-\frac{1}{2(d-2s)} \iint_{\R^d \times \R^d} \frac{u(x,t)u(y,t)}{|x-y|^{d-2s}} dxdy.
\end{align}
The competition between these two terms leads to finite-time blow-up and infinite time existence. As a matter of fact, \er{fracks} can be recast as
\begin{align}\label{chemical}
u_t=\nabla \cdot \left( u\nabla \mu  \right),
\end{align}
where $\mu$ is the chemical potential
\begin{align}
\mu=\frac{m}{m-1}u^{m-1}-c.
\end{align}
Multiplying \er{chemical} by $\mu$ and integrating it in space one obtains
\begin{align}
\frac{d F(u)}{dt}+\int_{\R^d} u|\nabla \mu|^2 dx=0.
\end{align}
This implies that $F[u(\cdot,t)]$ is non-increasing with respect to time.

Notice that equation \er{fracks} possesses a scaling invariance which leaves the $L^p$ norm invariant and produces a balance between the diffusion and the aggregation terms where
\begin{align}
p=\frac{d(2-m)}{2s}=m=\frac{2d}{d+2s}.
\end{align}
Indeed, if $u(x,t)$ is a solution to \er{fracks}, then $u_\lambda(x,t)=\lambda u \left( \lambda^{\frac{2-m}{2s}}x,\lambda^{m-1+\frac{2-m}{s}}t\right)$ is also a solution and this scaling preserves the $L^p$ norm $\|u_\lambda\|_{L^p(\R^d)}=\|u\|_{L^p(\R^d)}.$ For the critical case $m_c=2-2s/d$, the above scaling becomes the mass invariant scaling $u_\lambda(x,t)=\lambda u \left( \lambda^{1/d}x,\lambda^{1+\frac{2-2s}{d}}t\right)$. For the subcritical case $m>2-2s/d$, the aggregation dominates the diffusion for low density and prevents spreading. While for high density, the diffusion dominates the aggregation thus blow-up is precluded. On the contrary, for the supercritical case $1 \le m<2-2s/d,$ the aggregation dominates the diffusion for high density (large $\lambda$) and the density may have finite-time blow-up. While for low density (small $\lambda$), the diffusion dominates the aggregation and the density has infinite-time spreading. These behaviors also appear in many other physical systems such as thin film, Hele-Shaw, Stellar collapse as well as the nonlinear Schr\"{o}dinger equation \cite{BC99,Chandra,Wei83,Tom04}.

Our main goal in this paper is to classify the global existence and finite time blow-up of solutions by the initial data. The main tool for the analysis of \er{fracks} is the steady state which can be interpreted in the distributional sense
\begin{align}\label{steadystates}
\left\{
      \begin{array}{ll}
       \frac{m}{m-1}U_s^{m-1}-C_s=\overline{C}, ~~\mbox{in} ~~\Omega, \\
        U_s=0 \mbox{~ in ~} \R^d \setminus \Omega, \quad U_s>0 \mbox{ in } \Omega, \\
        C_s = \frac{1}{d-2s} \int_{\mathbb R^d} \frac{U_s(x)}{|x-y|^{d-2s}} dy, ~~\mbox{in}~~ \R^d.
      \end{array}\right.
\end{align}
Here $\overline{C}$ is any constant chemical potential and $\Omega=\{x \in \R^d ~\big|~ U_s(x)>0 \}$ is a connected open set in $\R^d$. Therefore, let us begin with the analysis of stationary states. We firstly stress that a lot of efforts have been put to investigate the steady states of \er{fracks} \cite{CCH17,CCH21,CHVY19,CHM18,CGH20}.

For $m>2-2s/d,$ this is the diffusion-dominated regime. There exists a unique up to translations compactly supported steady state to \er{fracks}. Moreover, such steady state is radially decreasing and it is the global minimiser of the energy functional $F(u).$ We refer to \cite{CHM18} for details. In the local case $s=1,$ the Riesz kernel $\frac{1}{|x|^{d-2s}}$ is replaced by Newtonian interaction for $d \ge 3$ and log interaction for $d=2.$ Analogous results for stationary states are provided in \cite{CHVY19,kimyao} for $m>2-2/d~(d \ge 3)$ and \cite{CCV15} for $m>1~(d=2).$

In the case $m=2-2s/d,$ this is the fair competition regime where the nonlinear diffusion and the nonlocal aggregation balance. The existence of the stationary states can only happen for a critical mass $M_c$ which is characterised by the optimal constant of a variant of the HLS inequality. $M_c$ turns out to be the only value of the mass for which the free energy $F(u)$ has minimizers and these minimizers are shown to be compactly supported, radially symmetric and non-increasing stationary solutions of equation \er{fracks}. See for instance \cite{CCH17} and the references therein. In the local case $s=1$, a clear dichotomy arises, that's, $M_c$ is the sharp critical mass such that if the initial mass $m_0<M_c,$ solutions of \er{fracks} exist globally in time, whereas there are blowing-up solutions for $m_0>M_c$. See \cite{BJ09} for the case $m=2-2/d~(d \ge 3)$ and \cite{BCM08,bdp06} for the case $m=1 ~(d=2)$. Other similar results for stationary states are given in the above two references as well.

For $1<m<2-2s/d,$ this is the aggregation-dominated regime. We will here focus on the case $m=\frac{2d}{d+2s}<2-\frac{2s}{d}$ and there are two reasons for us to concentrate on the special case $m=\frac{2d}{d+2s}$. On one hand, there are three subcases for stationary states in terms of $m$. For $\frac{2d}{d+2s}<m<2-\frac{2s}{d},$ all the stationary states are compactly supported. While they are not compactly supported for $1<m \le \frac{2d}{d+2s}$ and there is an explicit formula of one parameter family of stationary states in the case $m=\frac{2d}{d+2s}$ \cite{CLO06}. See Section \ref{Steady} for more details. Consequently, $m=\frac{2d}{d+2s}$ is the critical exponent separating steady states that have compactly support from those that do not. On the other hand, the free energy $F(u)$ is invariant under translations for $m=\frac{2d}{d+2s}.$ To be honest, plugging the invariant scaling $u_\lambda(x)=\lambda u\left( \lambda^{\frac{2-m}{2s}}x \right)$ into the free energy $F(u)$ we obtain
\begin{align}
F(u_\lambda)=\lambda^{\frac{(d+2s)m-2d}{2s}} F(u),
\end{align}
the free energy is invariant when $m=\frac{2d}{d+2s}$. For this reason $m$ is endowed with the energy critical exponent and denoted as $m_{ec}:=\frac{2d}{d+2s}.$

Indeed, for $m=\frac{2d}{d+2s},$ we can deduce that $\overline{C}=0$ and $\Omega=\R^d$ in \er{steadystates}, see Proposition \ref{Fusteady}. Hence, \er{steadystates} can be written as the following integral equation
\begin{align}\label{starstar}
C_s(x)& =\frac{1}{d-2s}\left( \frac{m-1}{m} \right)^{\frac{1}{m-1}} \int_{\R^d} \frac{C_s(y)^{\frac{1}{m-1}}}{|x-y|^{d-2s}} dy \nonumber \\
& =\frac{1}{d-2s}\left( \frac{d-2s}{2d} \right)^{\frac{d+2s}{d-2s}} \int_{\R^d} \frac{C_s(y)^{\frac{d+2s}{d-2s}}}{|x-y|^{d-2s}} dy.
\end{align}
It is proved by Chen and Li \cite{CLO06,CLO1,CLO2} that every positive solution $C_s$ of the integral equation \er{starstar} is radially symmetric and decreasing about some point $x_0$ and therefore assumes the form
\begin{align}\label{Cs}
C_s(x)=A \left( \frac{\lambda}{\lambda^2+|x-x_0|^2} \right)^{\frac{d-2s}{2}}
\end{align}
with some positive constants $A$ and $\lambda.$ Thus, we conclude that \er{fracks} uniquely admits a one-parameter family of stationary states
\begin{align}\label{Us}
U_s(\lambda)=B \left( \frac{\lambda}{\lambda^2+|x-x_0|^2} \right)^{\frac{d+2s}{2}}
\end{align}
with some point $x_0 \in \R^d$ and positive constants $B, \lambda$.
\begin{remark}
\begin{enumerate}
  \item The integral equation \er{starstar} is also closely related to the following family of semi-linear partial differential equations
      \begin{align}\label{230212}
       (-\Delta)^s w=w^{\frac{d+2s}{d-2s}}.
      \end{align}
      Apparently, equation \er{230212} is also of practical interest and importance. For instance, it arises as the Euler-Lagrange equation of the functional \cite{CLO06}
\begin{align}
I(w)=\frac{\int_{\R^d} \left| (-\Delta)^{\frac{s}{2}} w \right|^2 dx}{\left( \int_{\R^d} |w|^{\frac{2d}{d-2s}} dx \right)^{\frac{d-2s}{d}}}.
\end{align}
The classification of the solutions $C_s$ provides the best constant in the inequality of the critical sobolev embedding from $H^s(\R^d)$ to $L^{\frac{2d}{d-2s}}(\R^d)$:
\begin{align}
\left( \int_{\R^d} |w|^{\frac{2d}{d-2s}}dx \right)^{\frac{d-2s}{d}} \le \bar{K} \int_{\R^d} \left| (-\Delta)^{\frac{s}{2}} w \right|^2 dx
\end{align}
by $\bar{K}=\frac{\|(-\Delta)^{s/2} C_s \|_{L^2(\R^d)}^2}{ \|C_s\|_{L^{2d/(d-2s)}(\R^d)}^2}.$
  \item It is shown that all the solutions of partial differential equation \er{230212} satisfy the integral equation \er{starstar} and vise versa. Therefore, to classify the solutions of the partial differential equations, we only need to work on the integral equations \er{starstar}.
  \item Note that the second moment of $U_s$ is bounded for $s>1.$
\end{enumerate}
\end{remark}

We remark that the stationary solutions $U_s(\lambda)$ \er{Us} is the extremum for the HLS inequality \cite{lieb202} which states as follows.
\begin{lemma}(HLS inequality)
Let $q,q_1>1, d \ge 3$ and $0<\beta<d$ with $1/q+1/q_1+\beta/d=2.$ Assume $f \in L^q(\R^d)$ and $h \in L^{q_1}(\R^d)$. Then there exists a sharp constant $C(d,\beta,q)$ independent of $f$ and $h$ such that
\begin{align}\label{fh}
\left| \iint_{\R^d \times \R^d} \frac{f(x)h(y)}{|x-y|^{\beta}} dxdy   \right| \le C(d,\beta,q) \|f\|_{L^q(\R^d)} \|h\|_{L^{q_1}(\R^d)}.
\end{align}
If $q=q_1=\frac{2d}{2d-\beta},$ then
\begin{align}
C(d,\beta,q)=C(d,\beta)=\pi^{\beta/2} \frac{\Gamma\left( d/2-\beta/2 \right)}{\Gamma\left( d-\beta/2 \right)} \left( \frac{\Gamma(d/2)}{\Gamma(d)} \right)^{-1+\beta/d}.
\end{align}
In this case there is equality in \er{fh} if and only if $h \equiv (const.)f$ and
\begin{align}
f(x)=A \left( \gamma^2+|x-x_0|^2 \right)^{-(2d-\beta)/2}
\end{align}
for some $A>0, \gamma \in \R$ and $x_0 \in \R^d.$ Particularly, for $\beta=d-2s,$
\begin{align}\label{fhds}
f=U_s(\lambda)~~\mbox{for some}~~\lambda>0.
\end{align}
\end{lemma}

Recalling \er{steadystates} which reads as
\begin{align}
C_s(x)=\frac{m}{m-1} U_s^{m-1}(x),
\end{align}
the following equality from $F(u)$ holds:
\begin{align}\label{FUs}
F(U_s(x))& =\frac{1}{m-1}\| U_s \|_{L^m(\R^d)}^m-\frac{1}{2} \int_{\R^d} U_s C_s dx \nonumber \\
& = \frac{1}{m-1}\|U_s\|_{L^m(\R^d)}^m-\frac{m}{2(m-1)} \int_{\R^d} U_s^m dx \nonumber \\
&=\frac{2-m}{2(m-1)} \|U_s\|_{L^m(\R^d)}^m=\frac{2s}{d-2s}\|U_s\|_{L^m(\R^d)}^m.
\end{align}
Here we address that $\|U_s\|_{L^m(\R^d)}^m$ is a constant independent of $\lambda$ and $x_0.$ That is
\begin{align}\label{Lm}
\|U_s\|_{L^m(\R^d)}^m &=B^m \int_{\R^d} \left( \frac{\lambda}{\lambda^2+|x-x_0|^2} \right)^d dx \nonumber \\
&=B^m \pi^{\frac{d+1}{2}} 2^{1-d} \frac{1}{\Gamma \left( \frac{d+1}{2} \right)}.
\end{align}
Before proceeding further to show the dynamical behaviors of solutions, we define the weak solution which we will deal with in this paper.
\begin{definition}\label{weakdefine}(Weak and free energy solution)
Let $u_0$ be an initial condition satisfying \er{12} and $T \in (0,\infty].$
\begin{enumerate}
\item[\textbf{(i)}]
  A weak solution to \er{fracks} with initial data $u_0$ is a non-negative function $u \in L^\infty \left(0,T;L_+^1 \cap L^\infty(\R^d) \right)$ and for $\forall ~\psi \in C_0^\infty(\R^d)$ and $0<t<T$
 \begin{align}
 &\int_{\R^d} \psi u(\cdot,t)dx-\int_{\R^d} \psi u_0(x) dx =\int_0^t
 \int_{\R^d} \Delta \psi  u^m dx ds \nonumber \\
 & -\frac{1}{2} \int_0^t \iint_{\R^d\times \R^d}  \frac{[\nabla
 \psi(x)-\nabla \psi(y)] \cdot (x-y)}{|x-y|^2} \frac{u(x,s)
 u(y,s)}{|x-y|^{d-2s}} dxdy ds. \label{weak}
 \end{align}
\item[\textbf{(ii)}]
The weak solution $u$ is also a weak free energy solution to \er{fracks} if $u$ satisfies additional regularities that
  \begin{align}
   \nabla u^{m-1/2} \in L^2\left(0,T;L^2(\R^d)\right), \label{0201} \\
    u \in L^{3}\left(0,T;L^{\frac{3d}{d-2(1-2s)}}(\R^d) \right). \label{02011}
  \end{align}
Moreover, $F(u(\cdot,t))$ is a non-increasing function and satisfies
\begin{align}\label{Fuequality}
F[u(\cdot,t)]+\int_{0}^t \int_{\R^d} \left| \frac{2m}{2m-1}\nabla
u^{m-\frac{1}{2}}-\sqrt{u}\nabla c \right|^2dxds \leq
F(u_0)
\end{align}
  for all $t \in (0,T)$ with $c=\frac{1}{d-2s} \int_{\R^d} \frac{u(y)}{|x-y|^{d-2s}}dy.$
\end{enumerate}
\end{definition}
We emphasize that regularities \er{0201} and \er{02011} are enough to make sense of each term in \er{Fuequality}.
\begin{lemma}\label{uceps}
If $u \in L^{\frac{3d}{d-2(1-2s)}}(\R^d)$ and $c$ is expressed by \er{CC}, then
\begin{align}
\left\| u |\nabla c|^2 \right\|_{L^1(\R^d)} \leq C
\|u\|_{L^q(\R^d)}^3<\infty,~~q=\frac{3d}{d-2(1-2s)}\,.
\end{align}
\end{lemma}
\begin{proof}
By the H\"{o}lder inequality we have
\begin{align}\label{qqp1}
\left\| u |\nabla c|^2 \right\|_{L^1(\R^d)} \leq \|u\|_{L^q(\R^d)} \left\||\nabla
c|^2 \right\|_{L^{q'}(\R^d)},  ~~\frac{1}{q}+\frac{1}{q'}=1.
\end{align}
Then by the weak Young inequality \cite[formula (9), pp.
107]{lieb202}, it follows that
\begin{align}
\left\||\nabla c|^2\right\|_{L^{q'}(\R^d)} & =\|\nabla c\|_{L^{2q'}(\R^d)}^2=C
\left\|u(x)*\frac{x}{|x|^{d+2-2s}}\right\|_{L^{2q'}(\R^d)}^2 \nonumber \\
& \leq C
\|u\|_{L^{q}(\R^d)}^2 \left\|\frac{x}{|x|^{d+2-2s}}
\right\|_{L_w^{\frac{d}{d+1-2s}}(\R^d)}^2\le C\|u\|_{L^{q}(\R^d)}^2  ,
\end{align}
where $1+\frac{1}{2q'}=\frac{1}{q}+\frac{d+1-2s}{d}.$ Combining with
\er{qqp1} follows $q=\frac{3d}{d-2(1-2s)}$ and completes the proof. $\Box$
\end{proof}

Consequently, from \er{02011} one has that
\begin{align}
\int_{0}^T \left\| \sqrt{u} \nabla c \right\|_{L^2(\R^d)}^2 dt \leq
C\int_{0}^T \left\|u \right\|_{L^{3d/(d-2(1-2s))}(\R^d)}^3 dt<\infty.
\end{align}
So the term $\int_0^t \int_{\R^d} \left|\sqrt{u} \nabla c\right|^2 dx ds$ in \er{Fuequality} makes sense.

Thanks to \er{Lm}, we are now ready to state our main result by employing the steady states of \er{fracks}.
\begin{theorem}\label{main}
Let $d \ge 3$ and $m=\frac{2d}{d+2s}.$ We assume that $u_0(x) \in L^1_+ \cap L^\infty(\R^d)$ and $\int_{\R^d} |x|^2 u_0(x) dx <\infty,$ $U_s$ is the steady state of \er{fracks}. Suppose $F(u_0)<F(U_s).$
\begin{enumerate}
  \item[\textbf{(i)}] If $\|u_0\|_{L^m (\R^d)}<\|U_s\|_{L^m (\R^d)},$ then there exists a global weak solution to \er{fracks} satisfying that for any $0<t<\infty,$
     \begin{align}
      \|u(\cdot,t)\|_{L^1 \cap L^\infty(\R^d)} \le C\left( \|u_0\|_{L^1 \cap L^\infty(\R^d)} \right).
     \end{align}
     Furthermore, the weak solution is also a weak free energy solution satisfying the energy inequality \er{Fuequality}.
  \item[\textbf{(ii)}] If $\|u_0\|_{L^m (\R^d)}>\|U_s\|_{L^m (\R^d)},$ then the weak solution of \er{fracks} blows up in finite time $T$ in the sense that
     \begin{align}
       \displaystyle \limsup_{t \to T} \|u(\cdot,t)\|_{L^q(\R^d)}=\infty,~~\mbox{for any}~~1<q \le \infty.
     \end{align}
\end{enumerate}
\end{theorem}

Let us mention that the above results also hold true in the local setting $s=1$ for the lower dimensional case $d=2$, where Riesz potential is replaced by the log interaction and the energy critical exponent $\frac{2d}{d+2}$ and mass critical exponent $2-\frac{2}{d}$ coincides, i.e. $\frac{2d}{d+2}=2-\frac{2}{d}=1.$ In this case, the $L^1$ norm of the steady states is $8\pi$ and the global behaviors of solutions can be classified by the $L^1$ norm of the initial data. Actually, if $\|u_0\|_{L^1(\R^2)} < 8\pi,$ the solution of \er{fracks} exists globally in time, while the solution blows up in finite time for $\|u_0\|_{L^1(\R^2)} > 8\pi$ \cite{BCM08,bdp06}. Notice that in the nonlocal case $s>1,$ it is not straightforward to extend the above results to $d=2.$ In fact, there is a gap between $\frac{2d}{d+2s}$ and $2-\frac{2s}{d}$, i.e. $2-\frac{2s}{d}<\frac{2d}{d+2s}$ and the two dimensional degenerate problem is essentially new aspect.

The results are organised as follows. This work is entirely devoted to establish an exact criteria on the initial data for the global existence and finite time blow-up of dynamical solutions to \er{fracks}. With this aim we firstly construct the existence criterion which shows a key maximal existence time for the free energy solution of \er{fracks} in Section \ref{Existence}. Section \ref{Steady} mainly explores stationary solutions to \er{fracks} and shows that they are global minimizers of the free energy. Based on the above analysis, Section \ref{Proofmain} makes use of the steady states to prove the main theorem concerning the dichotomy, the global existence for $\|u_0\|_{L^m (\R^d)}<\|U_s\|_{L^m (\R^d)}$ and finite time blow-up for $\|u_0\|_{L^m (\R^d)}>\|U_s\|_{L^m (\R^d)}$. Finally, Section \ref{conclu} concludes the main work of this paper, some open questions for equation \er{fracks} are also addressed.

\section{Existence criterion} \label{Existence}

In this section, we will establish the existence criterion of a bounded weak solution to \er{fracks}. As in \cite{BL13,BL14,suku06}, we consider the regularized problem
\begin{align}\label{kseps}
\left\{
  \begin{array}{ll}
    \partial_t u_\varepsilon=\Delta u^m_\varepsilon+\varepsilon \Delta u_\varepsilon-\nabla \cdot \left( u_\varepsilon \nabla c_\varepsilon \right),~~x \in \R^d,~t>0, \\
 u_\varepsilon(x,0)=u_{0\varepsilon} \ge 0,
  \end{array}
\right.
\end{align}
where $c_\varepsilon$ is defined as
\begin{align}
c_\varepsilon=R_\varepsilon \ast u_\varepsilon
\end{align}
with the regularized Riesz potential
\begin{align}
R_\varepsilon(x)=\frac{1}{(d-2s)\left( |x|^2+\varepsilon^2  \right)^{\frac{d-2s}{2}}}.
\end{align}
Here $u_{0\varepsilon}$ is a sequence of approximation for $u_0$ and can be constructed to satisfy that there exists $\varepsilon_0>0$ such that for any $0<\varepsilon<\varepsilon_0,$
\begin{align}
\left\{
  \begin{array}{ll}
 u_{0\varepsilon} \ge 0, ~~ \|u_\varepsilon(x,0)\|_{L^1(\R^d)}=\|u_0\|_{L^1(\R^d)}, \\
 \int_{\R^d}|x|^2 u_{0\varepsilon} dx \to \int_{\R^d}|x|^2 u_0(x)dx,~~\mbox{as}~~\varepsilon \to 0, \\
 u_{0\varepsilon}(x) \to u_0(x)~~\mbox{in}~~L^q(\R^d),~~\mbox{for}~~1 \le q <\infty,~~\mbox{as}~~\varepsilon \to 0.
  \end{array}
\right.
\end{align}
This regularized problem has global in time smooth solutions for any $\varepsilon>0$. This approximation has been proved to be convergent. More precisely, following the arguments in \cite[Theorem 4.2]{BL14}, \cite[Lemma 4.8]{CHVY19} and \cite[Section 4]{suku06} we assert that if
\begin{align}\label{Linfinity}
\|u_\varepsilon(\cdot,t)\|_{L^\infty(\R^d)}<C_0,
\end{align}
where $C_0$ is independent of $\varepsilon>0,$ then there exists a subsequence $\varepsilon_n \to 0$ such that
\begin{align}
  \begin{array}{ll}
    u_{\varepsilon_n} \rightarrow u~~\mbox{in}~~L^r(0,T;L^r(\R^d)),~~1 \le r<\infty \label{conver1}
  \end{array}
\end{align}
and $u$ is a weak solution to \er{fracks} on $[0,T).$

According to the above analysis, a weak solution to \er{fracks} on $[0,T)$ exists when \er{Linfinity} is fulfilled. So we shall focus on establishing the availability of the $L^\infty$-bound. As we will see in the following theorem where the local in time existence and blow-up criteria are constructed, such a bound follows from the $L^r$ norm for $r>m=\frac{2d}{d+2s}$ which additionally provides a characterisation of the maximal existence time.

\begin{theorem}(Local in time existence and blow-up criteria)\label{ueps}
Under assumption \er{12} on the initial condition, there are a maximal existence time $T_w \in (0,\infty]$ and a weak solution $u$ to \er{fracks} on $[0,T_w)$. If $T_w<\infty$, then
\begin{align}
\|u(\cdot,t)\|_{L^m(\R^d)} \to \infty~~\mbox{as}~~t \to T_w.
\end{align}
\end{theorem}
\begin{proof}
To prove this result we need to refine the arguments used in the proof of \cite[Theorem 2.11]{BL13}. We follow a procedure analogous to the ones in \cite[Lemma 2.3]{BJ09}.

{\it\textbf{Step 1}} ($L^r$-estimates, $r \in (m,\infty)$) \quad It's obtained by multiplying the equation \er{kseps} with $ru_\varepsilon^{r-1}$ that
\begin{align}\label{230127}
& \frac{d}{dt} \int_{\R^d} u_\varepsilon^r dx+\frac{4mr(r-1)}{(m+r-1)^2} \int_{\R^d} \left| \nabla u_\varepsilon^{\frac{m+r-1}{2}} \right|^2 dx+\varepsilon \frac{4(r-1)}{r} \int_{\R^d} \left| \nabla u_\varepsilon^{r/2} \right|^2 dx \nonumber \\
=& (2s-2)(r-1) \iint_{\R^d \times \R^d} \frac{u_\varepsilon^r(x) u_\varepsilon(y)}{\left( |x-y|^2+\varepsilon^2 \right)^{\frac{d+2-2s}{2}}} dxdy.
\end{align}
Due to $2s>2,$ we treat the right hand side of \er{230127} by applying the HLS inequality \er{fh} with $f=u_\varepsilon^r(x)$ and $h=u_\varepsilon(y)$ such that
\begin{align}\label{cdsr}
\iint_{\R^d \times \R^d} \frac{u_\varepsilon^r(x) u_\varepsilon(y)}{\left( |x-y|^2+\varepsilon^2 \right)^{\frac{d+2-2s}{2}}} dxdy & \le \iint_{\R^d \times \R^d} \frac{u_\varepsilon^r(x) u_\varepsilon(y)}{|x-y|^{d+2-2s}} dxdy \nonumber \\
& \le C_{dsr} \|u_\varepsilon^r\|_{L^{q}(\R^d)} \|u_\varepsilon\|_{L^{q_1}(\R^d)} = C_{dsr} \left\|u_\varepsilon \right\|_{L^{\frac{r+1}{1+(2s-2)/d}}(\R^d)}^{r+1},
\end{align}
where $\frac{1}{q}+\frac{1}{q_1}+\frac{d+2-2s}{d}=2$ and we have used $q_1=rq=\frac{r+1}{1+(2s-2)/d}$, $C_{dsr}$ is a bounded constant depending on $d,s,r$. Furthermore, for
$$
1 \le q_0<r_0<\frac{2d}{d-2},
$$
we have the following GNS inequality \cite{bp07}: there exists a positive constant $C$ such that
\begin{align}
\|v\|_{L^{r_0}(\R^d)} \le C \|\nabla v\|_{L^{2}(\R^d)}^\theta \|v\|_{L^{q_0}(\R^d)}^{1-\theta},\quad \theta=\frac{\frac{1}{q_0}-\frac{1}{r_0}}{\frac{1}{q_0}-\frac{d-2}{2d} },
\end{align}
which we apply with $v=u_\varepsilon^{\frac{m+r-1}{2}}$ to obtain
\begin{align}\label{230129}
\left\|u_\varepsilon \right\|_{L^{\frac{r+1}{1+(2s-2)/d}}(\R^d)}^{r+1} \le C \left\|\nabla u_\varepsilon^{\frac{m+r-1}{2}} \right\|_{L^2(\R^d)}^{a \theta} \|u_\varepsilon\|_{L^{q_0 (m+r-1)/2}(\R^d)}^{a(1-\theta)\frac{m+r-1}{2}},
\end{align}
where the parameters satisfy
\begin{align*}
a \frac{m+r-1}{2}=r+1, \quad r_0 \frac{m+r-1}{2}=\frac{r+1}{1+(2s-2)/d}.
\end{align*}
Here we pick
\begin{align*}
m=\frac{d(2-m)}{2s}< q_0 \frac{m+r-1}{2}
\end{align*}
and simple computations show that
\begin{align*}
a \theta=2 \left( 1+ \frac{ \frac{2-m}{(m+r-1)q_0}-\frac{2s}{2d}}{ \frac{1}{q_0}-\frac{d-2}{2d} } \right)<2
\end{align*}
in case of $1 < m < 2-2s/d$. By Young's inequality, from \er{230129} we continue to obtain
\begin{align}
\left\|u_\varepsilon \right\|_{L^{\frac{r+1}{1+(2s-2)/d}}(\R^d)}^{r+1} \le \frac{2mr}{C_{dsr} (2s-2) (m+r-1)^2} \left\|\nabla u_\varepsilon^{\frac{m+r-1}{2}} \right\|_{L^2(\R^d)}^2 +C \|u_\varepsilon\|_{L^{q_0 (m+r-1)/2}(\R^d)}^{(r+1)(1-\theta)/(1-a\theta/2)},
\end{align}
where the constant $C_{dsr}$ is defined as in \er{cdsr} and
\begin{align}
\frac{(r+1)(1-\theta)}{1-a\theta/2}=\frac{2s-(2-m)(2s-2)+2sr-d(2-m)}{2s-\frac{2d(2-m)}{q_0 (m+r-1)}}>q_0 \frac{m+r-1}{2}
\end{align}
in case of
$$\frac{d(2-m)}{2s}<q_0 \frac{m+r-1}{2} \le r.$$
In particular, taking
$$
q_0 \frac{m+r-1}{2}=r
$$
we write
\begin{align}\label{0131}
(2s-2)C_{dsr}(r-1) \|u_\varepsilon\|_{L^{\frac{r+1}{1+(2s-2)/d}}(\R^d)}^{r+1} \le & \frac{2mr(r-1)}{(m+r-1)^2} \left\|\nabla u_\varepsilon^{\frac{m+r-1}{2}} \right\|_{L^2(\R^d)}^2 \nonumber \\
& + C \left( \|u_\varepsilon\|_{L^{r}(\R^d)}^{r} \right)^{1+\frac{\frac{2s}{d}-(2-m)\frac{2s-2}{d}}{\frac{2s}{d}r-(2-m)}}
\end{align}
for $r>\frac{d(2-m)}{2s}=m$ due to $m=\frac{2d}{d+2s}.$ Therefore, substituting \er{cdsr} and \er{0131} into \er{230127} we thus end up with
\begin{align}\label{230131}
& \frac{d}{dt} \|u_\varepsilon \|_{L^r(\R^d)}^r +\frac{2mr(r-1)}{(m+r-1)^2} \int_{\R^d} \left| \nabla u_\varepsilon^{\frac{m+r-1}{2}} \right|^2 dx+\varepsilon \frac{4(r-1)}{r} \int_{\R^d} \left| \nabla u_\varepsilon^{r/2} \right|^2 dx \nonumber \\
\le & C \left( \|u_\varepsilon\|_{L^{r}(\R^d)}^{r} \right)^{1+\frac{\frac{2s}{d}-(2-m)\frac{2s-2}{d}}{\frac{2s}{d}r-(2-m)}}.
\end{align}
Hence the local in time $L^r$-estimates are followed
\begin{align}\label{Lr}
\|u_\varepsilon (\cdot,t)\|_{L^r(\R^d)}^r & \le \frac{1}{\left( \|u_\varepsilon(0) \|_{L^r(\R^d)}^{-r \delta}-\delta C(d,s,r) t \right)^{\frac{1}{\delta}}}\nonumber \\
&= \left(\frac{C(m,s,r)}{T_0-t}\right)^{\frac{1}{\delta}},\qquad T_0=\frac{\|u_\varepsilon(0) \|_{L^r(\R^d)}^{-r \delta}}{\delta C(d,s,r)},
\end{align}
where $\delta=\frac{1-\frac{(2-m)(2s-2)}{2s}}{r-\frac{d(2-m)}{2s}}>0$ for $r>\frac{d(2-m)}{2s}$ because of $1 < m<2-\frac{2s}{d}.$ Moreover, coming back to \er{230131}, we further deduce that
\begin{align}\label{nablaLr}
\left\|\nabla u_\varepsilon^{\frac{m+r-1}{2}} \right\|_{L^2(0,T_0;L^2(\R^d))} \le C\left( \|u_\varepsilon(0)\|_{L^r(\R^d)} \right).
\end{align}

{\it\textbf{Step 2}} ($L^m$-estimates)\quad The local in time $L^m$-estimate is derived from \er{Lr} by the interpolation inequality
\begin{align}\label{LLm}
\|u_\varepsilon\|_{L^m(\R^d)} & \le \|u_\varepsilon\|_{L^r(\R^d)}^\eta \|u_\varepsilon\|_{L^1(\R^d)}^{1-\eta} \nonumber\\
& \le M^{1-\eta} \left(\frac{C(m,s,r)}{T_0-t} \right)^{\frac{\eta}{r \delta}}, \qquad \eta=\frac{1-\frac{1}{m}}{1-\frac{1}{r}}.
\end{align}

{\it\textbf{Step 3}} ($L^\infty$-estimates)\quad The fact \er{Lr} allows us to imitate the proof of \cite[Theorem 2.3]{CW19} and \cite[Theorem 4.2]{BL14} to conclude that
\begin{align}\label{Linfinity23}
\displaystyle \sup_{0<t<T_0} \|u_\varepsilon\|_{L^\infty(\R^d)} \le C\left( \|u_\varepsilon(0)\|_{L^1(\R^d)},\|u_\varepsilon(0)\|_{L^\infty(\R^d)} \right)
\end{align}
by virtue of Moser iterative method. Thus following similar arguments of \cite[Section 4]{suku06} with the aid of regularities \er{nablaLr} and \er{Linfinity23}, we deduce that there exists a weak solution $u$ of \er{fracks} by passing to the limit $u_\varepsilon \to u$ as $\varepsilon \to 0$ (without relabeling). Finally, a direct consequence of \cite[Theorem 2.4]{BJ09} characterises the maximal existence time of the weak solution and completes the proof. $\Box$
\end{proof}

Furthermore, a detailed analysis of the proof of \cite[Theorem 2.11]{BL13} shows that the weak solution is also a free energy solution satisfying $F(u(\cdot,t)) \le F(u_0).$

\begin{proposition}(Existence of the free energy solution)\label{energy}
Under assumption \er{12} on the initial data and $u_\varepsilon(\cdot,t) \in L^\infty \left(0,T_w;L^\infty(\R^d) \right)$ on the approximated sequence, there exists a free energy solution to \er{fracks} on $[0,T_w).$
\end{proposition}
\begin{proof}
Multiplying $\mu_\varepsilon=\frac{m}{m-1}u_\varepsilon^{m-1}-c_\varepsilon$ to the regularized equation \er{kseps} yields
\begin{align}\label{230202}
& \frac{d}{dt}F\left(u_\varepsilon(\cdot,t)\right)+\int_{\R^d} \left
| \frac{2m}{2m-1}\nabla
u_\varepsilon^{m-\frac{1}{2}}-\sqrt{u_\varepsilon}\nabla
c_\varepsilon \right |^2 dx +\frac{4\varepsilon}{m} \int_{\R^d}
\left|\nabla u_\varepsilon^{m/2}\right|^2 dx \nonumber \\
=& \varepsilon (2s-2) \iint_{\R^d \times \R^d} \frac{u_\varepsilon(x)u_\varepsilon(y)}{\left( |x-y|^2+\varepsilon^2 \right)^{\frac{d+2-2s}{2}}} dxdy
\end{align}
for any $t \in [0,T_w).$ Integrating \er{230202} in time from $0$ to $t$ arrives at
\begin{align}\label{doublestar}
& F\left[u_\varepsilon(\cdot,t)\right]+\int_0^t \int_{\R^d} \left
| \frac{2m}{2m-1}\nabla
u_\varepsilon^{m-\frac{1}{2}}-\sqrt{u_\varepsilon}\nabla
c_\varepsilon \right |^2 dxds \nonumber \\
\le & F\left[u_\varepsilon(0)\right]+\varepsilon (2s-2) \int_0^t \iint_{\R^d \times \R^d} \frac{u_\varepsilon(x)u_\varepsilon(y)}{\left( |x-y|^2+\varepsilon^2 \right)^{\frac{d+2-2s}{2}}} dxdyds
\end{align}
From \er{nablaLr} and \er{Linfinity23} we know that for any $t \in [0,T_w)$, the following basic estimates hold true:
\begin{align}
&\|u_\varepsilon\|_{L^\infty \left(0,t; L_+^1 \cap L^\infty(\R^d)\right)} \le C, \label{1234} \\
&\left\|\nabla u_\varepsilon^{\frac{m+r-1}{2}} \right\|_{L^2(0,t; L^2(\R^d))} \le C,~~\mbox{for}~~1 \le r<\infty. \label{12345}
\end{align}
Here $C$ represent constants depending only on $\|u_\varepsilon(0)\|_{L_+^1 \cap L^\infty(\R^d)}$. We shall use \er{1234} and \er{12345} to pass to the limit $\varepsilon \to 0$ in \er{doublestar}.

Firstly, the term in the right hand side of \er{doublestar} is bounded with the help of \er{1234} and the HLS inequality \er{fh} that
\begin{align}\label{230207}
\int_0^t \iint_{\R^d \times \R^d} \frac{u_\varepsilon(x)u_\varepsilon(y)}{\left( |x-y|^2+\varepsilon^2 \right)^{\frac{d+2-2s}{2}}} dxdyds \le C \int_0^t \|u_\varepsilon \|_{L^{2d/(d-2+2s)}(\R^d)}^2 ds \le C
\end{align}
for any $t \in [0,T_w).$

Secondly, we note that the dissipation term is uniformly bounded
\begin{align}
\int_0^{t} \int_{\R^d} \left
| \frac{2m}{2m-1}\nabla
u_\varepsilon^{m-\frac{1}{2}}-\sqrt{u_\varepsilon}\nabla
c_\varepsilon \right |^2 dx ds \le C.
\end{align}
Actually, taking $r=m$ in \er{12345} yields $\left\|\nabla u_\varepsilon^{m-1/2} \right\|_{L^2(0,t;L^2(\R^d))} \le C$. Besides, by Lemma \ref{uceps} and using \er{1234} we find
\begin{align}
\int_0^t \left\| u_\varepsilon |\nabla c_\varepsilon|^2 \right\|_{L^1(\R^d)} ds \le C \int_0^t \|u_\varepsilon\|_{L^{\frac{3d}{d-2(1-2s)}}(\R^d)}^3 ds \le C.
\end{align}
Now by going through similar steps as the proof of \cite[Theorem 2.11]{BL13}, it is straightforward from the convergence properties \er{conver1} to pass to the limit $\varepsilon \to 0$ (without relabeling) and thus obtain
the lower semi-continuity of the dissipation term
\begin{align}\label{dissipationae}
& \int_{0}^t \int_{\R^d} \left|\frac{2m}{2m-1}\nabla
u^{m-1/2}-\sqrt{u}\nabla c \right|^2 dx ds \nonumber \\
\le & \displaystyle
\liminf_{\varepsilon \to 0} \int_{0}^t \int_{\R^d}
\left|\frac{2m}{2m-1}\nabla
u_\varepsilon^{m-1/2}-\sqrt{u_\varepsilon}\nabla c_\varepsilon
\right|^2 dxds.
\end{align}

In addition, the convergence of the free energy can be directly deduced from \er{conver1} that
\begin{align}\label{Fstrong}
& F(u_\varepsilon(\cdot,t))= \frac{1}{m-1} \int_{\R^d} u_\varepsilon^m dx-\frac{1}{2(d-2s)} \iint_{\R^d \times \R^d} \frac{u_\varepsilon(x,t)u_\varepsilon(y,t)}{|x-y|^{d-2s}} dxdy \to \nonumber \\
& \frac{1}{m-1} \int_{\R^d} u^m dx-\frac{1}{2(d-2s)} \iint_{\R^d \times \R^d} \frac{u(x,t)u(y,t)}{|x-y|^{d-2s}} dxdy=F(u(\cdot,t))~~\mbox{a.e.~in}~(0,T_w).
\end{align}
Hence combining \er{230207}, \er{dissipationae} and \er{Fstrong}, letting $\varepsilon \to 0$ in \er{doublestar} follows that
\begin{align}
&F \left[ u(\cdot,t) \right]+ \int_{0}^t \int_{\R^d}\left |
\frac{2m}{2m-1}\nabla u^{m-\frac{1}{2}}-\sqrt{u}\nabla c \right
|^2dxds \nonumber\le F[u_0(x)],  \quad a.e. ~~t \in (0,T_w).
\end{align}
Thus $u$ is a free energy solution satisfying the free energy inequality \er{Fuequality}. $\Box$
\end{proof}

\section{Steady states} \label{Steady}

As we have just seen, the existence time of the free energy solution to \er{fracks} heavily depends on the behavior of its $L^m$-norm. Moreover, the first term of the free energy involves the $L^m$-norm and \er{fhds} tells us that the second term of the free energy is related to the stationary solution $U_s$ to \er{fracks}. Hence, in order to prove the existence of a global weak solution by the energy functional, the information given by steady states will be of paramount importance. Therefore, in this section, we will give a detailed analysis of steady states to \er{fracks}.

The steady equation of \er{fracks} is followed in the sense of distribution
\begin{align}\label{11}
    \Delta U_s^m-\nabla \cdot \left( U_s(x) \nabla C_s(x) \right)=0,~~ x \in \R^d,
\end{align}
where $C_s(x)$ is given by
\begin{align}\label{22}
C_s(x)=\frac{1}{d-2s} \int_{\R^d} \frac{U_s(y)}{|x-y|^{d-2s}} dy.
\end{align}
Starting from a Pohozaev type identity, we will show a full characterization of minimizers of $F(u)$ to relate to the steady states.
\begin{lemma}(A Pohozaev identity for steady solutions) Assume $U_s \in L^m(\R^d)$ satisfying \er{11}, then the steady solution satisfies the following identity in the sense of distribution
\begin{align}\label{Poho}
\int_{\R^d} U_s^m dx=\frac{d-2s}{2d} \int_{\R^d} C_s U_s dx.
\end{align}
\end{lemma}
\begin{proof}
We follow a similar argument as done in \cite{bdp06,bp07}. Consider a cut-off function $\psi_R(x) \in C_0^\infty(\R^d)$ and $\psi_R(x)=|x|^2$ for $|x|<R, ~\psi_R(x)=0$ for $|x|\ge 2R.$ Then we compute
\begin{align}\label{33}
\int_{\R^d} \Delta U_s^m \psi_R(x) dx=\int_{\R^d} U_s^m \Delta \psi_R(x) dx
\end{align}
and
\begin{align}\label{44}
& \int_{\R^d} \psi_R(x) \nabla \cdot  (U_s \nabla C_s)dx=-\int_{\R^d} \nabla \psi_R(x) \cdot U_s \nabla C_s dx \nonumber \\
=& \frac{1}{2}\iint_{\R^d \times \R^d} \frac{[\nabla \psi_R(x)-\nabla \psi_R(y)]\cdot (x-y)}{|x-y|^2} \frac{U_s(x)U_s(y)}{|x-y|^{d-2s}} dxdy.
\end{align}
Both terms in the right-hand side of \er{33} and \er{44} are bounded (because $\Delta \psi_R(x)$ and $\frac{[\nabla \psi_R(x)-\nabla \psi_R(y)]\cdot (x-y)}{|x-y|^2}$ are bounded and $U_s\in L^m(\R^d)$). Therefore, as $R \to \infty,$ we may pass to the limit in each term using the Lebesgue monotone convergence theorem and obtain the identity. $\Box$
\end{proof}

With this identity in hand, we are now ready to make use of the variational structure of the free energy to deduce the following results.
\begin{proposition}\label{Fusteady}(Critical point of $F(u)$)
Let $\bar{u} \in L_+^1 \cap L^\infty(\R^d)$ with $\|\bar{u}\|_{L^1(\R^d)}=M$ be a global minimizer of $F(u)$. Then $\bar{u}$ satisfies the Euler-Lagrange condition
\begin{align}\label{230213}
\frac{m}{m-1} \bar{u}^{m-1}-\frac{1}{d-2s} \bar{u} \ast \frac{1}{|x|^{d-2s}}=\overline{C},~~\forall~x \in supp(\bar{u}),
\end{align}
where the constant
$$
 \overline{C}=\frac{1}{M}\left( \frac{m}{m-1}\|\bar{u}\|_{L^m(\R^d)}^m-\frac{1}{d-2s} \iint_{\R^d \times \R^d} \frac{\bar{u}(x)\bar{u}(y)}{|x-y|^{d-2s}}dxdy \right).
$$
Furthermore, such minimizer is the steady state $U_s$ of \er{fracks} in the distributional sense \er{11}. Therefore, for $m=\frac{2d}{d+2s},$ $supp(U_s)=\R^d.$
\end{proposition}
\begin{proof}
We use some ideas from \cite{BL13,CCV15,CHM18}. Set $Y_M=\left\{ u \in L_+^1 \cap L^\infty(\R^d), \|u\|_{L^1(\R^d)}=M  \right\}$. Let $\bar{u} \in Y_M$ be a global minimizer of $F(u)$ with $\Omega=supp(\bar{u}).$ For any $\psi \in C_0^\infty(\Omega),$ denoting
\begin{align}\label{varphi}
\varphi=\left( \psi-\frac{1}{M} \int_{\Omega} \psi \bar{u} dx \right) \bar{u}(x),
\end{align}
then $supp (\varphi) \subseteq \Omega$ and $\int_{\Omega} \varphi dx=0.$ Besides, there exists
\begin{align*}
\varepsilon_0: = \frac{ \displaystyle \min_{y \in  supp\varphi} \bar{u}(y)} { \displaystyle \max_{y \in supp \varphi} \left| \varphi (y) \right| }  >0
\end{align*}
such that $\bar{u}+\varepsilon \varphi \ge 0$ in $\Omega$ for $0<\varepsilon<\varepsilon_0.$ Now $\bar{u}$ is a global minimizer of $F(u)$ in $\Omega$ if and only if
\begin{align}
\frac{d}{d \varepsilon} \Big |_{\varepsilon=0}F(\bar{u}+\varepsilon \varphi)=0.
\end{align}
The above definition implies that for any $\int_{\Omega} \varphi(x)dx=0$
\begin{align}\label{3star}
\int_{\Omega} \left( \frac{m}{m-1} \bar{u}^{m-1}-\frac{1}{d-2s} \int_{\R^d} \frac{\bar{u}(y)}{|x-y|^{d-2s}}dy \right) \varphi(x) dx=0.
\end{align}
Then by the definition \er{varphi} of $\varphi,$ \er{3star} can be read as
\begin{align}
\int_{\Omega} \left( \frac{m}{m-1} \bar{u}^{m-1}-\frac{1}{d-2s} \int_{\R^d} \frac{\bar{u}(y)}{|x-y|^{d-2s}} dy -\overline{C} \right) \bar{u} \psi dx=0,~~\mbox{for any}~\psi \in C_0^\infty(\Omega).
\end{align}
This gives
\begin{align}\label{2tri}
\frac{m}{m-1} \bar{u}^{m-1}-\frac{1}{d-2s} \int_{\R^d} \frac{\bar{u}(y)}{|x-y|^{d-2s}} dy=\overline{C},~~a.e.~~\mbox{in}~\Omega.
\end{align}
Since $\bar{u} \in L_+^1 \cap L^\infty(\R^d),$ we have from similar arguments of \cite[Lemma 2.3]{CHVY19} that $\bar{u}$ is continuous in $\Omega$. Owning the regularity we arrive at
\begin{align}\label{fivestar}
\nabla \cdot \left( \nabla \bar{u}^m-\frac{1}{d-2s} \bar{u} \nabla  \bar{u} \ast \frac{1}{|x|^{d-2s}} \right)=\nabla \cdot \left(\bar{u} \nabla \left( \frac{m}{m-1}\bar{u}^{m-1}-\frac{1}{d-2s} \bar{u} \ast \frac{1}{|x|^{d-2s}}  \right) \right)=0
\end{align}
in the sense of distribution. Therefore, all global minimizers of the free energy are stationary solutions of \er{fracks} in the distributional sense characterized by \er{fivestar}. Then by employing Theorem 2.2 of \cite{CHVY19}, the steady states must be radially decreasing up to translations.

Now we declare that the steady state $U_s$ of \er{fracks} satisfies
\begin{align}\label{313}
\frac{m}{m-1} U_s^{m-1}-\frac{1}{d-2s} \int_{\R^d}   \frac{U_s(y)}{|x-y|^{d-2s}} dy=\overline{C},~~\forall~x \in \Omega=supp(U_s).
\end{align}
Keeping \er{Poho} in mind, taking inner product to \er{313} by $U_s$ in $\R^d$ yields
\begin{align}
\overline{C}=\frac{1}{M} \left( \frac{1}{m-1}-\frac{d+2s}{d-2s} \right) \int_{\R^d} U_s^m dx.
\end{align}
When $m=\frac{2d}{d+2s}, \overline{C}=0,$ thus from \er{313} we claim that $\Omega=\R^d.$ If not, choose a point $x_0 \in \partial \Omega,$ as we take a sequence of point $x_n \to x_0$ with $x_n \in \Omega,$ we have that $\frac{m}{m-1} U_s^{m-1}(x_n) \to 0,$ whereas the sequence $\frac{1}{d-2s} \int_{\R^d} \frac{U_s(y)}{|x_n-y|^{d-2s}} dy> 0$, a contradiction. $\Box$
\end{proof}

Now we are ready to assert that the steady states of \er{fracks} can be expressed explicitly. By Proposition \ref{Fusteady} one has that
\begin{align}
\left\{
  \begin{array}{ll}
    \frac{m}{m-1} U_s^{m-1}-C_s=0,~~\mbox{in}~~\R^d, \\
    C_s=\frac{1}{d-2s} U_s \ast \frac{1}{|x|^{d-2s}}.
  \end{array}
\right.
\end{align}
Hence $C_s$ is endowed with the following integral equation
\begin{align}\label{866}
C_s(x)& =\frac{1}{d-2s}\left( \frac{m-1}{m} \right)^{\frac{1}{m-1}} \int_{\R^d} \frac{C_s(y)^{\frac{1}{m-1}}}{|x-y|^{d-2s}} dy \nonumber \\
& =\frac{1}{d-2s}\left( \frac{d-2s}{2d} \right)^{\frac{d+2s}{d-2s}} \int_{\R^d} \frac{C_s(y)^{\frac{d+2s}{d-2s}}}{|x-y|^{d-2s}} dy.
\end{align}
It has been proved in \cite{CLO06} that every positive solution of \er{866} is radially symmetric and decreasing about some point $x_0 \in \R^d$ and therefore assumes the form
\begin{align}
C_s(x)= c \left(\frac{\lambda}{\lambda^2+|x-x_0|^2}\right)^{\frac{d-2s}{2}}
\end{align}
with some positive constants $c$ and $\lambda.$ It gives our final estimates for this part.
\begin{proposition}\label{Steadyexpression}
For $m=\frac{2d}{d+2s},$ any steady state to \er{fracks} is radially symmetric up to translation and given by a one-parameter family for some $\lambda>0$ and $c>0$
\begin{align}
U_s(x)= c \left(\frac{\lambda}{\lambda^2+|x-x_0|^2}\right)^{\frac{d+2s}{2}} \end{align}
with $\int_{\R^d}U_s^m(x)dx=C(d).$
\end{proposition}

\section{Proof of Theorem \ref{main}} \label{Proofmain}

With the aid of the characterization of steady states and the free energy, we can proceed to distinguish the two cases $\|u_0\|_{L^m(\R^d)}<\|U_s\|_{L^m(\R^d)}$ and $\|u_0\|_{L^m(\R^d)}>\|U_s\|_{L^m(\R^d)}$ to show $\|u(\cdot,t)\|_{L^m(\R^d)}$ can be bounded from below or above separately, then we establish the global existence and finite time blow-up of solutions to \er{fracks}.

\subsection{Proof of the finite time blow-up (ii) of Theorem \ref{main}}

In this subsection, we start with the case $\|u_0\|_{L^m(\R^d)}>\|U_s\|_{L^m(\R^d)}$ and show the finite time blow-up of the weak solution to \er{fracks}. Firstly we prepare a priori bound.
\begin{lemma}\label{blowupbdd}(A priori bound)
Under assumption \er{12}, if the initial free energy $F(u_0)<F(U_s)$ and $\|u_0\|_{L^m(\R^d)}>\|U_s\|_{L^m(\R^d)}$, then there exists a constant $\mu>1$ such that the corresponding free energy solution $u$ satisfies that for any $t>0$
\begin{align}
\|u(\cdot,t)\|_{L^m(\R^d)} > \mu \|U_s\|_{L^m(\R^d)}.
\end{align}
\end{lemma}
\begin{proof}
Firstly letting
\begin{align*}
f=u(x), h=u(y), \beta=d-2s, q=q_1=\frac{2d}{2d-\beta}
\end{align*}
in the HLS inequality \er{fh}, we can infer from the expression of $F(u)$ that
\begin{align}\label{111}
F(u) &=\frac{1}{m-1} \|u\|_{L^m(\R^d)}^m -\frac{1}{2(d-2s)} \iint_{\R^d \times \R^d} \frac{u(x)u(y)}{|x-y|^{d-2s}} dxdy \nonumber\\
& \ge \frac{1}{m-1} \|u\|_{L^m(\R^d)}^m-\frac{C(d,\beta)}{2(d-2s)} \|u\|_{L^m(\R^d)}^2.
\end{align}
While \er{fhds} also tells us that
\begin{align}\label{222}
F(U_s)=\frac{1}{m-1} \|U_s\|_{L^m(\R^d)}^m-\frac{C(d,\beta)}{2(d-2s)} \|U_s\|_{L^m(\R^d)}^2.
\end{align}
We define an auxiliary function
$$
g(x)=\frac{1}{m-1}x-\frac{C(d,\beta)}{2(d-2s)} x^{2/m}.
$$
It can be seen from \er{111} and \er{222} that
$$
g \left( \|U_s\|_{L^m(\R^d)}^m \right)=F(U_s)\ge g\left( \|u\|_{L^m(\R^d)} ^m \right)
$$
for all $u \in L^m(\R^d).$ Namely, the maximum point $x_0$ of $g(x)$ is attained by
\begin{align}\label{333}
x_0=\left( \frac{m(d-2s)}{(m-1)C(d,\beta) } \right)^{\frac{m}{2-m}}=\|U_s\|_{L^m(\R^d)}^m.
\end{align}
In addition, in the case that $F(u_0)<F(U_s),$ there is a $0<\delta<1$ such that
\begin{align}\label{delta}
F(u_0)<\delta F(U_s).
\end{align}
Thus, by the monotonicity of $F(u)$ with respect to time, it follows that
\begin{align}\label{444}
g\left( \|u\|_{L^m(\R^d)}^m \right) \le F(u) \le F(u_0)<\delta F(U_s)=\delta g\left( \|U_s\|_{L^m(\R^d)}^m \right).
\end{align}
Therefore, for any $x>\|U_s\|_{L^m(\R^d)}^m, g(x)$ is strictly decreasing and it has a strictly decreasing inverse function $g^{-1}.$ So if $\|u_0\|_{L^m(\R^d)}^m>\|U_s\|_{L^m(\R^d)}^m,$ one has that for some $\mu>1$ depending on $\delta,$
\begin{align}\label{555}
\|u(\cdot,t)\|_{L^m(\R^d)}^m>\mu \|U_s\|_{L^m(\R^d)}^m.
\end{align}
This completes this lemma. $\Box$
\end{proof}

When $\|u_0\|_{L^m(\R^d)}<\|U_s\|_{L^m(\R^d)}$, following the lines \er{111}-\er{444} in the proof of Lemma \ref{blowupbdd}, we can deduce the following reverse result.
\begin{corollary}\label{coro1}
Under assumption \er{12}, if the initial free energy $F(u_0)<F(U_s)$ and $\|u_0\|_{L^m(\R^d)}<\|U_s\|_{L^m(\R^d)}$, then there exists a constant $\mu_1<1$ such that the corresponding free energy solution $u$ satisfies that for any $t>0$
\begin{align}
\|u(\cdot,t)\|_{L^m(\R^d)} < \mu_1 \|U_s\|_{L^m(\R^d)}.
\end{align}
\end{corollary}

Now we can utilize the standard argument relying on the evolution of the second moment of solutions as originally done in \cite{JL92} to prove the finite time blow-up result in Theorem \ref{main}.
\begin{proposition}(Finite time blow-up)
Assume $\int_{\R^d} |x|^2 u_0(x) dx<\infty, F(u_0)<F(U_s)$ and $\|u_0\|_{L^m(\R^d)}>\|U_s\|_{L^m(\R^d)}$. Let $u(x,t)$ be a weak free energy solution to \er{fracks} on $[0,T_w),$ then it satisfies
\begin{align}
\frac{d}{dt}\int_{\R^d} |x|^2 u(x,t) dx=-4s \int_{\R^d} u^m dx+2(d-2s)F(u),~0<t<T_w.
\end{align}
Here $T_w<\infty$ and $\|u\|_{L^m(\R^d)}$ blows up in finite time.
\end{proposition}
\begin{proof}
Here we show the formal computation, the passing to the limit from the approximated problem \er{kseps} can be done on account of \cite[Theorem 2.11]{BL13} and \cite[Lemma 6.2]{S06} without any further complication. By integrating by parts in \er{fracks} and symmetrising the second term, the time derivative of the second moment of the weak solution is endowed with
\begin{align}
\frac{d}{dt}m_2(t) &=\frac{d}{dt} \int_{\R^d} |x|^2 u(x,t) dx=2d \int_{\R^d} u^m dx-\iint_{\R^d \times \R^d} \frac{u(x,t)u(y,t)}{|x-y|^{d-2s}} dxdy \nonumber \\
&=\left( 2d-\frac{2}{m-1}(d-2s) \right)\int_{\R^d} u^m dx+2(d-2s)F(u) \nonumber \\
&=-4s \int_{\R^d} u^m dx+2(d-2s) F(u).
\end{align}
Then applying Lemma \ref{blowupbdd} with some $\mu>1$ and the decreasing of $F(u)$ in time indicates
\begin{align}
\frac{d}{dt} \int_{\R^d} |x|^2 u(x,t) dx & \le -4s \mu \|U_s\|_{L^m(\R^d)}^m+2(d-2s) F(u_0) \nonumber \\
&< -4s \mu \|U_s\|_{L^m(\R^d)}^m+2(d-2s) F(U_s) \nonumber \\
&=4s (1-\mu) \|U_s\|_{L^m(\R^d)}^m<0,
\end{align}
where we have used the identity  $F(U_s)=\frac{2s}{d-2s}\|U_s\|_{L^m(\R^d)}^m$ which can be derived by
\begin{align}
F(U_s) &=\frac{1}{m-1} \int_{\R^d} U_s^m dx-\frac{1}{2} \int_{\R^d} U_s C_s dx \nonumber \\
&= \frac{1}{m-1} \int_{\R^d} U_s^m dx-\frac{m}{2(m-1)} \int_{\R^d} U_s^m dx \nonumber \\
& =\frac{2-m}{2(m-1)} \int_{\R^d} U_s^m dx=\frac{2s}{d-2s}\|U_s\|_{L^m(\R^d)}^m
\end{align}
due to $C_s=\frac{m}{m-1}U_s^{m-1}.$ Thus there exists a $T>0$ such that $\displaystyle \lim_{t \to T} m_2(t)=0$.

On the other hand, it follows from the H\"{o}lder inequality that
\begin{align}
&\int_{\R^d} u(x) dx=\int_{|x| \le R} u(x) dx+\int_{|x|>R} u(x) dx \nonumber \\
\le & C R^{(m-1)d/m} \|u\|_{L^m(\R^d)} +\frac{1}{R^2} m_2(t).
\end{align}
Choosing $R=\left( \frac{C m_2(t)}{\|u\|_{L^m(\R^d)}} \right)^{\frac{1}{a+2}}$ with $a=\frac{(m-1)d}{m}$ gives
\begin{align}
\|u_0\|_{L^1(\R^d)}=\|u\|_{L^1(\R^d)} \le c \|u\|_{L^m(\R^d)}^{\frac{2}{a+2}} m_2(t)^{\frac{a}{a+2}}.
\end{align}
It follows that
\begin{align}
\displaystyle \limsup_{t \to T} \|u(\cdot,t)\|_{L^m(\R^d)} \ge \displaystyle \lim_{t \to T} \frac{\|u_0\|_{L^1(\R^d)}^{(a+2)/2}}{c m_2(t)^{a/2}}=\infty.
\end{align}
Thus the proof is completed. $\Box$
\end{proof}

\subsection{Proof of the global existence (i) of Theorem \ref{main}}

A direct consequence of Theorem \ref{ueps} and Corollary \ref{coro1} is the following global existence result.
\begin{proposition}(Global existence)
Under assumption \er{12}, assume $F(u_0)<F(U_s)$ and $\|u_0\|_{L^m(\R^d)}<\|U_s\|_{L^m(\R^d)}$, then there exists a free energy solution to \er{fracks} on $[0,\infty).$
\end{proposition}
\begin{proof}
By Theorem \ref{ueps}, there are $T_w$ and a free energy solution to \er{fracks} in $[0,T_w)$ with initial data $u_0.$ We then infer from Corollary \ref{coro1} that there exists a $\mu_1<1$, for all $t \in [0,T_w)$,
$$
\|u(\cdot,t)\|_{L^m(\R^d)} < \mu_1 \|U_s\|_{L^m(\R^d)}.
$$
Therefore, we can deduce that $u \in L^\infty \left(0,\min(T,T_w);L^m(\R^d)\right)$ for every $T>0$ which implies that $T_w=\infty$ by Theorem \ref{ueps}. $\Box$
\end{proof}

\section{Conclusions}\label{conclu}

This paper concerns equation \er{fracks} for the energy critical exponent $m=\frac{2d}{d+2s}$. It is proved that any steady states $U_s(x)$ of \er{fracks} is radially symmetric and decreasing about some point $x_0 \in \R^d$ and assumes the from $U_s(x)=c \left( \frac{\lambda}{\lambda^2+|x-x_0|^2} \right)^{\frac{d+2s}{2}}$ for some positive constants $c$ and $\lambda$. Our results show that the $L^p$ norm of the stationary solutions plays a critical role on determining the global existence and finite time blow-up of solutions to \er{fracks}. Here $p=\frac{d(2-m)}{2s}=m$ and the $L^p$ norm produces a balance between the diffusion and the aggregation terms under the invariant scaling for equation \er{fracks}. Actually, $\|U_s\|_{L^p(\R^d)}$ is a constant only depending on $d,m$ and we prove that the $L^p$ norm of the stationary solutions $\|U_s\|_{L^p(\R^d)}$ is the sharp condition on the initial data separating infinite-time existence from finite-time blow-up. Precisely, in the case that the initial free energy $F(u_0)$ is less than $F(U_s)$ which is the free energy of the stationary solutions, there exists a global weak solution satisfying the free energy inequality when the initial data satisfies $\|u_0\|_{L^m(\R^d)}<\|U_s\|_{L^m(\R^d)}$ and the solution blows up at finite time provided by the condition $\|u_0\|_{L^m(\R^d)}>\|U_s\|_{L^m(\R^d)}$. For the critical case $\|u_0\|_{L^m(\R^d)}=\|U_s\|_{L^m(\R^d)},$ whether the solution exists globally in time or blows up in finite time is still unknown.


\end{document}